\documentclass[a4paper,11pt]{amsart}
\usepackage{amsmath,amssymb,amsfonts}
\usepackage{mathrsfs,latexsym,amsthm,stmaryrd,enumerate,epic,mathrsfs,dsfont}
\usepackage{amscd}
\usepackage[all]{xy}
\begin{document}
\title[Wedderburn's basis]{Categorification
of Wedderburn's\\ basis  for  $\mathbb{C}[S_n]$}

\author{Volodymyr Mazorchuk}
\address{V. M.: Department of Mathematics, Uppsala University (Sweden).}
\email{mazor\symbol{64}math.uu.se}
\author{Catharina Stroppel}
\address{C. S.: Department of Mathematics, Glasgow University
(United Kingdom).} \email{c.stroppel\symbol{64}maths.gla.ac.uk}
\thanks{The first author was partially supported by STINT,
the Royal Swedish Academy of Sciences, and the Swedish Research Council,
the second author by EPSRC}

\numberwithin{equation}{section}

\newtheorem{proposition}{Proposition}
\newtheorem{lemma}[proposition]{Lemma}
\newtheorem{corollary}[proposition]{Corollary}
\newtheorem{theorem}[proposition]{Theorem}
\newtheorem{definition}[proposition]{Definition}
\newtheorem{conjecture}[proposition]{Conjecture}
\newtheorem{example}[proposition]{Example}
\newtheorem{remark}[proposition]{Remark}
\newcommand{\oplusop}[1]{{\mathop{\oplus}\limits_{#1}}}
\newcommand{\oplusoop}[2]{{\mathop{\oplus}\limits_{#1}^{#2}}}
\newcommand{\cS}{\mathcal{S}}
\newcommand{\cH}{\mathcal{H}}
\newcommand{\ff}{\footnote}
\renewcommand{\to}{\rightarrow}
\newtheorem{theoremintro}{Theorem}
\renewcommand{\thetheoremintro}{\Roman{theoremintro}}

\font\sc=rsfs10 at 12 pt
\font\scs=rsfs10 at 10 pt
\font\scb=rsfs10 at 16 pt
\font\scbb=rsfs10 at 18 pt

\newcommand{\ccC}{\mathscr{C}}
\newcommand{\ccS}{\mathscr{S}}

\def\la{\lambda}
\def\op{\operatorname}
\def\C{\mathbb C}
\def\R{\mathbb R}
\def\N{\mathbb N}
\def\Z{\mathbb Z}
\def\Q{\mathbb Q}
\def\g{\mathfrak g}
\def\p{\mathfrak p}
\def\h{\mathfrak h}
\def\n{\mathfrak n}
\def\mmm{\mathbf m}
\def\mmn{\mathbf n}
\newcommand{\mC}{\mathbb{C}}
\newcommand{\Ext}{\operatorname{Ext}}
\newcommand{\End}{\operatorname{End}}

\newcommand{\add}{\operatorname{add}}
\newcommand{\Ann}{\operatorname{Ann}}

\def\F{\mathbb F}
\def\S{\mathbb S}
\def\l{\lbrace}
\def\r{\rbrace}
\def\o{\otimes}
\def\lra{\longrightarrow}
\newcommand{\ba}{\mathbf{a}}
\newcommand{\cA}{\mathcal{A}}
\newcommand{\cB}{\mathcal{B}}
\newcommand{\cL}{\mathcal{L}}
\newcommand{\dmod}{\mathrm{-mod}}
\newcommand{\gmod}{\mathrm{-gmod}}
\newcommand{\mc}{\mathcal}
\newcommand{\mZ}{\mathbb{Z}}
\newcommand{\tto}{\twoheadrightarrow}
\newcommand{\mg}{\mathfrak{g}}
\newcommand{\mh}{\mathfrak{h}}
\newcommand{\ma}{\mathfrak{a}}
\newcommand{\mb}{\mathfrak{b}}
\newcommand{\cU}{\mathcal{U}}
\def\cF{\mathcal{F}}
\def\Hom{\textrm{Hom}}
\def\drawing#1{\begin{center} \epsfig{file=#1} \end{center}}
\def\mc{\mathcal}
\def\mf{\mathfrak}
\def\mb{\mathbb}

\def\yesnocases#1#2#3#4{\left\{ \begin{array}{ll} #1 & #2 \\ #3 & #4
\end{array} \right. }

\newcommand{\define}{\stackrel{\mbox{\scriptsize{def}}}{=}}
\def\hsm{\hspace{0.05in}}

\def\cO{\mathcal{O}}   
\def\cC{\mathscr{C}}
\def\sln{\mathfrak{sl}(n)}

\begin{abstract}
M.~Neunh{\"o}ffer studies in \cite{Ne} a certain basis of $\mathbb{C}[S_n]$
with the origins in \cite{Lu} and shows that this basis is in fact Wedderburn's
basis, hence decomposes the right regular representation of $S_n$
into a direct sum of irreducible representations (i.e. Specht or
cell modules). In the present paper we rediscover essentially the same basis
with a categorical origin coming from projective-injective modules in certain
subcategories of the BGG-category $\mathcal{O}$. Inside each of these
categories, there is a dominant projective module which plays a crucial role
in our arguments and will additionally be used to show that {\it Kostant's
problem} (\cite{Jo}) has a negative answer for some simple highest weight
module over the Lie algebra $\mathfrak{sl}_4$. This disproves the general
belief that Kostant's problem should have a positive answer for all simple
highest weight modules in type $A$.
\end{abstract}
\maketitle


\section{The main result}\label{s1}

Let $n$ be a positive integer and $S_n$ the group of permutations of
the elements from $\{1,2,\dots,n\}$. Denote by $\cS$ the usual set of Coxeter
generators of $S_n$ and by $\mathcal{H}=\mathcal{H}(S_n,\cS)$ the
associated (generic) Iwahori-Hecke algebra. The algebra
$\mathcal{H}$ is a free $\mathbb{Z}[v,v^{-1}]$-module with basis
$\{H_w\vert w\in S_n\}$ and multiplication given by
\begin{displaymath}
H_xH_y=H_{xy} \text{ if } l(x)+l(y)=l(xy) \,\text{ and }\,
H_s^2=H_e+(v^{-1}-v)H_s \text{ for }s\in S,
\end{displaymath}
where $l:S_n\to\mathbb{Z}$ denotes the length function with respect
to $\cS$. Denote by $\{\underline{H}_w\vert w\in S_n\}$ the {\it
Kazhdan-Lusztig basis} (in the normalization of \cite{So}). We also
denote by $\{\hat{\underline{H}}_w \vert w\in S_n\}$ the dual
Kazhdan-Lusztig  basis of $\mathcal{H}$, defined via
$\tau(\hat{\underline{H}}_v\underline{H}_{w^{-1}})=\delta_{v,w}$, where
$\tau$ is the standard symmetrizing trace form.

The group algebra $\mathbb{C}[S_n]$ of $S_n$ is obtained by
specializing $v$ to $1$ in $\cH$, more precisely: by extending
first the scalars in $\cH$ to $\mC$ and then  factoring out
the ideal generated by $v-1$ we get an epimorphism of $\mC$-algebras,
which we call the {\em evaluation map}:
\begin{displaymath}
\mathrm{ev}:\mathbb{C}\otimes_{\mathbb{Z}}\mathcal{H}
\overset{\mathrm{proj}}{\tto}\left(\mathbb{C}\otimes_{\mathbb{Z}}
\mathcal{H}\right)/(v-1)
\overset{\sim}{\to}
\mathbb{C}[S_n],\quad 1\otimes H_w\mapsto w.
\end{displaymath}

The Robinson-Schensted correspondence (see e.g. \cite[3.1]{Sa})
defines a bijection between elements $w\in S_n$ and pairs
$(a(w),b(w))$ of standard tableaux with $n$ boxes, such that $a(w)$ and
$b(w)$ are of the same shape. For every element $w\in S_n$ we denote by
$\mathtt{R}_w=\{x\in S_n\mid a(x)=a(w)\}$ the {\it right cell} of
$S_n$ which contains $w$. Let $\overline{w}$ denote the unique involution
in $\mathtt{R}_w$. Beside $a(\overline{w})=a(w)$ the element $\overline{w}$
satisfies (and is characterized by the property) $a(\overline{w})=b(\overline{w})$. It is the
{\it Duflo involution} of $\mathtt{R}_w$.

Our main result is the construction of a basis $\{f_w\vert w\in S_n\}$
of $\mC[S_n]$ compatible with its regular right $S_n$-module structure
in the following way:

\begin{theorem}\label{tmain}
For $w\in S_n$ set
$f_w=\mathrm{ev}(\hat{\underline{H}}_{\overline{w}}\underline{H}_w)$.
Then the following holds:
\begin{enumerate}[(a)]
\item The elements $\{f_w\vert w\in S_n\}$ form a basis of
$\mathbb{C}[S_n]$.
\item Let $x\in S_n$ and consider the linear
span ${\bf S}(x)$ of all $f_w$, $w\in \mathtt{R}_x$. Then ${\bf
S}(x)$ is invariant with respect to the right action of
$\mathbb{C}[S_n]$ and isomorphic to the (irreducible) cell module
associated with $\mathtt{R}_x$.
\end{enumerate}
\end{theorem}

In other words, there is a decomposition of the right regular representation of
$S_n$ into a direct sum of irreducible modules which is compatible with the
basis $\{f_w\vert w\in S_n\}$. In fact the theorem and its proof are valid
over any field of characteristic zero. As an example, for $n=3$ let
$s$ and $t$ be the simple reflections, then our basis
consists of the elements
\begin{displaymath}
\begin{array}{rclrcl}
f_e&=&(e-s-t+st+ts-sts)e=e-s-t+st+ts-sts,\\
f_s&=&(s-ts-st+sts)(s+e)=e+s-t-ts,\\
f_t&=&(t-ts-st+sts)(t+e)=e+t-s-st,\\
f_{st}&=&(s-ts-st+sts)(st+s+t+e)=s+st-ts-sts,\\
f_{ts}&=&(t-ts-st+sts)(ts+s+t+e)=t+ts-st-sts,\\ 
f_{sts}&=&sts(e+t+s+st+ts+sts)=e+t+s+st+ts+sts.
\end{array}
\end{displaymath}
Unfortunately, this method does not give a basis for the algebra
$\mathcal{H}$.

Theorem~\ref{tmain} turns out to be related to the paper \cite{Ne},
where a similar basis was studied. Let $\{\mathtt{R}_i\,:\, i\in I\}$ be a
set of right cells in $S_n$ containing exactly one representative of
each two-sided sell. For each $i\in I$ and $(x,y)\in \mathtt{R}_i\times
\mathtt{R}_i$ set
$h^i_{(x,y)}=\mathrm{ev}(\hat{\underline{H}}_{x^{-1}}\underline{H}_y)$. From
\cite{Ne} it follows that $\{h^i_{(x,y)}:i\in I,\, (x,y)\in \mathtt{R}_i\times
\mathtt{R}_i\}$ has properties analogous to those of the basis
$\{f_w\vert w\in S_n\}$ from Theorem~\ref{tmain}. Moreover,
in \cite{Ne} it is even proved that a
normalized version of $\{h^i_{(x,y)}\vert i\in I,(x,y)\in \mathtt{R}_i\times
\mathtt{R}_i\}$ is in fact Wedderburn's basis of $\mathbb{C}[S_n]$
(i.e. basis elements correspond to matrix units in Wedderburn's
decomposition of $\mathbb{C}[S_n]$). The origins of the
basis $\{h^i_{(x,y)}\vert i\in I,(x,y)\in \mathtt{R}_i\times \mathtt{R}_i\}$ go
further back to \cite{Lu}. There is an asymptotic version $\mathcal{J}$ of the
Hecke algebra, introduced by Lusztig in \cite{Lu} together with a homomorphism
$\Psi: \mathcal{H}\rightarrow \mathbb{Z}[v,v^{-1}]\otimes_{\mathbb{Z}}
\mathcal{J}$ which becomes an isomorphism over $\mathbb{Q}(t)$. As pointed out
to us by Neunh{\"o}ffer, the basis $\{h^i_{(x,y)}\vert i\in I,(x,y)\in
\mathtt{R}_i\times \mathtt{R}_i\}$  is exactly Lusztig's basis for
$\mathcal{J}$ pulled back via the homomorphism $\Psi$ to $\mathcal{H}$. The
connection to \cite{Ne} is the following:

\begin{theorem}\label{tmain2}
$\{f_w\vert w\in S_n\}$=$\{h^i_{(x,y)}\vert i\in I,(x,y)\in
\mathtt{R}_i\times \mathtt{R}_i\}$.
\end{theorem}

The origins of Theorem~\ref{tmain}, as well as the proof of
Theorem~\ref{tmain2}, are categorical; and this is absolutely crucial for our
arguments. In particular, our setup is completely different from the
combinatorial approach of \cite{Ne}. There are alternative combinatorial
approaches to the construction of a basis for $\mathbb{C}[S_n]$
and some related algebras in which the regular representation decomposes
into a direct sum of irreducibles, see \cite{RW}, \cite{Mu2},
\cite{Mu}, \cite{Mat1}, \cite{Mat2}. There are also alternative combinatorial
constructions (e.g. \cite{KL}, \cite{Al1}, \cite{Al2}) giving decompositions
of the regular representation of  $S_n$ into irreducible representations
using an explicit basis, which lead only to filtrations whose successive
subquotients are irreducible.

\noindent {\bf Acknowledgment.} We thank Ken Brown for suggestions, Meinfold Geck for information about \cite{Ne}, and
Michael Rapoport for helpful discussions. We also thank
Max Neunh{\"o}ffer, Susumu Ariki and Andrew Mathas for remarks on a preliminary version of the paper. Finally, we thank
the referee for very useful comments and suggestions.

\section{Proof of Theorem~\ref{tmain}}\label{s2}
We prove Theorem~\ref{tmain} by giving an explicit categorical
interpretation of all ingredients, which is based on
the categorification of cell modules as established in
\cite[Section~4]{MS} (the original idea of categorifying the Hecke
algebra goes back to \cite{KL} and \cite{BG}). The main players here
are certain subquotient categories of the famous BGG category
$\mathcal{O}$ (for the latter see \cite{BGG}).

Let $\mathcal{O}_0$ be the principal block of $\cO$ for the simple
complex Lie algebra $\mathfrak{sl}_n$ with its standard triangular
decomposition. The simple objects in $\mathcal{O}_0$ are the $L(w)$,
$w\in S_n$, the simple highest weight modules with the highest
weight $w(\rho)-\rho$, where $\rho$ is the half-sum of all positive
roots. Let $\Delta(w)$ and $P(w)$ denote the Verma and the
indecomposable projective module with unique simple quotient
isomorphic to $L(w)$ respectively. Further, denote by $\theta_w$ the
indecomposable projective endofunctor of $\mathcal{O}_0$ with the
property $\theta_w P(e)\cong \theta_w \Delta(e)\cong P(w)$ (see
\cite{BG}). Finally, let $[\mathcal{O}_0]$ denote the complexified
Grothendieck group of $\mathcal{O}_0$. For $M\in \mathcal{O}_0$ we
denote by $[M]$ its image in $[\mathcal{O}_0]$.

There is a $\mC$-linear isomorphism $\varphi:[\mathcal{O}_0]\to
\mathbb{C}[S_n]$ with $\varphi([\Delta(w)])=w$. The Kazhdan-Lusztig conjecture
(\cite{KL}, proved in \cite{BeBe}, \cite{BK}) implies that
$\varphi([P(w)])=\mathrm{ev}(\underline{H}_w)$  (for an overview see e.g.
\cite[Subsection~3.4]{MS}). The standard bilinear form on $\mathbb{C}[S_n]$ is
categorified via the bifunctor $\mathrm{Ext}^*({}_-,{}_-)$
(\cite[Section~5]{KMS} or \cite[Subsection~4.6]{MS}). Indecomposable projective
and simple modules form dual bases with respect to this form, and hence
\begin{eqnarray}\label{eq01}
\label{phi}
\varphi([L(w)])&=&\mathrm{ev}(\hat{\underline{H}}_w)
\end{eqnarray}
The functors $\theta_w$ are exact and induce therefore $\mathbb{C}$-linear
endomorphisms $[\theta_w]$ of $[\mathcal{O}_0]$. By \cite[Theorem~3.4(iv)]{BG}
and \cite{So} (for a more adjusted reformulation see \cite[Subsection~3.4]{MS})
we have
\begin{equation}\label{eq1}
\varphi([\theta_w M])=\varphi([\theta_w][M])=
\varphi([M])\mathrm{ev}(\underline{H}_w).
\end{equation}
for all $M$ in $\cO_0$. Recall the right cells mentioned above and let $\leq_{\mathtt{R}}$
be the right preorder on $S_n$. Fix $w\in W$ and set
$\hat{\mathtt{R}}_w=\{x\in S_n\vert x\leq_{\mathtt{R}}y \text{ for some }
y\in \mathtt{R}_w\}$.
Associated with the right cell $\mathtt{R}_w$ of $w$ we have the full
subcategory $\mathcal{O}_0^{\hat{\mathtt{R}}_w}$ of $\mathcal{O}_0$, which
consists of all modules $M$ with all composition subquotients of the form
$L(x)$ with $x\in \hat{\mathtt{R}}_w$. Let
$\mathrm{Z}^{\hat{\mathtt{R}}_w}:  \mathcal{O}_0\to
\mathcal{O}_0^{\hat{\mathtt{R}}_w}$
be the natural projection
functor which takes the maximal quotient that lies in
$\mathcal{O}_0^{\hat{\mathtt{R}}_w}$. All this is built up such that
we have
\begin{equation}\label{eq2}
\mathrm{Z}^{\hat{\mathtt{R}}_w}\theta_x\cong \theta_x
\mathrm{Z}^{\hat{\mathtt{R}}_w}
\end{equation}
for any $x,w\in S_n$, (\cite[Lemma~19]{MS}). For $x\in S_n$ we
define $P^{\hat{\mathtt{R}}_w}(x)= \mathrm{Z}^{\hat{\mathtt{R}}_w}
P(x)$, and it follows that
\begin{equation}\label{eq3}
 P^{\hat{\mathtt{R}}_w}(x)\neq
0\text{ if and only if } x\in \hat{\mathtt{R}}_w.
\end{equation}
Moreover, the set $\{P^{\hat{\mathtt{R}}_w}(x)\vert x\in
\hat{\mathtt{R}}_w\}$ constitutes a complete list of indecomposable
projective modules in $\mathcal{O}_0^{\hat{\mathtt{R}}_w}$.

The following provides a basis of $\mathbb{C}[S_n]$ with most of the desired properties:

\begin{proposition}\label{lemma1}
For $w\in S_n$  define
$g_w=\varphi([P^{\hat{\mathtt{R}}_w}(w)])\in\mathbb{C}[S_n]$. Then the
following holds:
\begin{enumerate}[(a)]
\item $\{g_w\vert w\in S_n\}$ is a basis of $\mathbb{C}[S_n]$.
\item For every $x\in S_n$ the linear span of
$\{g_w\vert w\in \mathtt{R}_x\}$  is invariant
with respect to the right action of $S_n$ and is isomorphic to the
cell module associated with $\mathtt{R}_x$.
\end{enumerate}
\end{proposition}

\begin{proof}
As $|\{g_w\vert w\in S_n\}|=|S_n|=\dim_{\mathbb{C}} \mathbb{C}[S_n]$,
it is enough to show that the elements from $\{g_w\vert w\in S_n\}$ are
linearly independent. By definition of the category
$\cO_0^{\hat{\mathtt{R}}_x}$,
all the simple composition factors of $P^{\hat{\mathtt{R}}_x}(w)$ are of the
form $L(z)$ where $z$ is smaller or equal to $x$ in the right cell
order. Therefore, when expressed in the specialization
$\{\operatorname{ev}(\hat{\underline{H}}_z)\mid z\in S_n\}$
of the dual Kazhdan-Lusztig basis, the element $g_w$ is a
linear combination of basis elements, corresponding to $z\in
\hat{\mathtt{R}}_x$ (see \eqref{eq01}).
By induction on the right order, it is then enough to show
that for any $x\in S_n$ the elements from $\{g_w\vert w\in \mathtt{R}_x\}$ are
linearly independent. By \cite[Theorem~1]{KMS} and \cite[Theorem~18]{MS}, these
elements form the Kazhdan-Lusztig basis in the cell module associated with
$\mathtt{R}_x$. The cell module is a subquotient of $\mathbb{C}[S_n]$. Hence
these elements are linearly independent already in $\mathbb{C}[S_n]$. The first
statement follows.

To prove the invariance it is enough to show, thanks to \eqref{eq1},
that projective functors preserve the additive subcategory
$\mathscr{A}$ of $\mathcal{O}_0^{\hat{\mathtt{R}}_x}$ generated by
the indecomposable projective modules $P^{\hat{\mathtt{R}}_x}(w)$,
$w\in \mathtt{R}_x$. Since $\mathcal{H}$ is generated by the
$\underline{H}_s$, where $s$ runs through $\cS$, it is enough to
show that for any $s\in\cS$ and $w\in \mathtt{R}_x$ the module
$\theta_s P^{\hat{\mathtt{R}}_x}(w)$ belongs to $\mathscr{A}$. Now
\eqref{eq2}, \cite[(4.1)]{MS}  and \eqref{eq3} provide the following
three isomorphisms:
\begin{eqnarray*}
\theta_s P^{\hat{\mathtt{R}}_x}(w)
=\theta_s\mathrm{Z}^{\hat{\mathtt{R}}_x} \theta_w\Delta(e)&\cong&\mathrm{Z}^{\hat{\mathtt{R}}_x}\theta_s\theta_w  \Delta(e)\\
&\cong& \mathrm{Z}^{\hat{\mathtt{R}}_x}\left(\oplus_{y\geq_{\mathtt{R}}
w}\theta_y^{m_y}  \Delta(e)\right)=\oplus_{y\geq_{\mathtt{R}} \mathtt{R}_x}
(\mathrm{Z}^{\hat{\mathtt{R}}_x}P(y))^{\oplus {m_y}}\\ &\cong&
\oplus_{y\in \mathtt{R}_x} \oplus_{i=1}^{m_y}
P^{\hat{\mathtt{R}}_w}(y)
\end{eqnarray*}
for some non-negative integers $m_y$.
The claim about the invariance follows. The claim about the cell
module follows from \cite[Theorem~16 and Theorem~18]{MS}.
\end{proof}

Now Theorem~\ref{tmain} follows from the
following statement:

\begin{proposition}\label{prop3}
We have $f_w=g_w$ for all $w\in S_n$. In particular,
Theorem~\ref{tmain} follows from Proposition~\ref{lemma1}.
\end{proposition}

\begin{proof}
We already know that $\varphi([L(w)])=\operatorname{ev}(\hat{\underline{H}}_w)$
for all $w\in S_n$. Thanks to \eqref{eq1} and the definitions of $f_w$ and
$g_w$, the proposition is implied by the

\begin{center}
\frame{{ \it Key statement:}\,\,\,  Let $w\in S_n$, then $\theta_w
L(\overline{w})\cong P^{\hat{\mathtt{R}}_w}(w)$,}
\end{center}

which also explains the categorical meaning of the basis.
In what follows we prove this statement.\\

Recall that $P^{\hat{\mathtt{R}}_w}(w)\cong \theta_w
P^{\hat{\mathtt{R}}_w}(e)$ by \eqref{eq2}. To prove the key
statement we have to study the dominant projective module
$P^{\hat{\mathtt{R}}_w}(e)$ in $\mathcal{O}_0^{\hat{\mathtt{R}}_w}$
in more detail.

\begin{lemma}\label{lemma4}
Let $x\in \mathtt{R}_w$ be such that $x\neq \overline{w}$. Then
$[P^{\hat{\mathtt{R}}_w}(e):L(x)]=0$.
\end{lemma}

\begin{proof}
Recall that the functor $\theta_x$ is both left and right adjoint to the
functor $\theta_{x^{-1}}$. Hence we have
\begin{displaymath}
\begin{array}{rcl}
[P^{\hat{\mathtt{R}}_w}(e):L(x)]&=&
\dim \mathrm{Hom}_{\mathcal{O}}(P^{\hat{\mathtt{R}}_w}(x),
P^{\hat{\mathtt{R}}_w}(e))\\
&=&
\dim \mathrm{Hom}_{\mathcal{O}}(\theta_x P^{\hat{\mathtt{R}}_w}(e),
P^{\hat{\mathtt{R}}_w}(e))\\
&=&
\dim \mathrm{Hom}_{\mathcal{O}}(P^{\hat{\mathtt{R}}_w}(e),
\theta_{x^{-1}} P^{\hat{\mathtt{R}}_w}(e)).\\
\end{array}
\end{displaymath}

As $x\neq \overline{w}$, we have $x\neq x^{-1}$, and hence, using
\cite[Theorem~3.6.6]{Sa}, we have $a(x^{-1})=b(x)\neq  a(x)$.
Thus $x^{-1}\not\in \mathtt{R}_w$. Since $a(x^{-1})$ and $a(x)$
still have the same shape, it follows that $x^{-1}\not\in
\hat{\mathtt{R}}_w$ (\cite[Exercise~10, page 198]{BjBr}). Therefore
$\theta_{x^{-1}} P^{\hat{\mathtt{R}}_w}(e)= \theta_{x^{-1}}
\mathrm{Z}^{\hat{\mathtt{R}}_w} \Delta(e)\cong
\mathrm{Z}^{\hat{\mathtt{R}}_w} \theta_{x^{-1}} \Delta(e)\cong
\mathrm{Z}^{\hat{\mathtt{R}}_w} P(x^{-1})=0$
and thus $\dim
\mathrm{Hom}_{\mathcal{O}}(P^{\hat{\mathtt{R}}_w}(e),
\theta_{x^{-1}} P^{\hat{\mathtt{R}}_w}(e))=0$ as well.
\end{proof}

\begin{lemma}\label{lemma5}
For any $x\in \mathtt{R}_w$ and $y\in \hat{\mathtt{R}}_w\setminus\mathtt{R}_w$
we have $\theta_xL(y)=0$. In particular,
$[P^{\hat{\mathtt{R}}_w}(e):L(\overline{w})]>0$.
\end{lemma}

\begin{proof}
As $P^{\hat{\mathtt{R}}_w}(y)\tto L(y)$ and $\theta_x$ is exact, we have
$\theta_x P^{\hat{\mathtt{R}}_w}(y)\tto \theta_x L(y)$. Applying \eqref{eq2}
we even have that $\theta_x L(y)$ is a homomorphic image of the module
$\mathrm{Z}^{\hat{\mathtt{R}}_w}\theta_x\theta_y \Delta(e)$.

Note that $\theta_x L(y)\in \mathcal{O}_0^{\hat{\mathtt{R}}_y}$,
in particular, all simple subquotients of $\theta_x L(y)$ have the form
$L(z)$, $z\in \hat{\mathtt{R}}_y$.

On the other hand, it follows from \cite[(4.1)]{MS} that
$\theta_x\theta_y$ is a direct sum of functors of the form
$\theta_z$, where $z\geq_{\mathtt{L}} x$. Hence, by \eqref{eq3}, all
simple quotients of the module
$\mathrm{Z}^{\hat{\mathtt{R}}_w}\theta_x\theta_y \Delta(e)$ have the
form $L(x)$. As $x\not\in \hat{\mathtt{R}}_y$ by our choice of $y$,
we must have  $\theta_x L(y)=0$.

We know that $P^{\hat{\mathtt{R}}_w}(\overline{w})=
\theta_{\overline{w}}P^{\hat{\mathtt{R}}_w}(e)\neq 0$. By
Lemma~\ref{lemma4} and the above,  $L(\overline{w})$ is the only
subquotient of $P^{\hat{\mathtt{R}}_w}(e)$ which has the chance not to
be annihilated by $\theta_{\overline{w}}$. Altogether we must have
$[P^{\hat{\mathtt{R}}_w}(e):L(\overline{w})]>0$
\end{proof}

\begin{lemma}\label{lemma6}
$[P^{\hat{\mathtt{R}}_w}(e):L(\overline{w})]=1$.
\end{lemma}

\begin{proof}
Assume for a moment that $\mathtt{R}_w$ contains an element of the
form $w'_0w_0$, where $w_0$ is the longest element of $S_n$ and
$w'_0$ is the longest element of some parabolic (Young) subgroup $W$
of $S_n$. Let $S$ be the set of simple reflections in $W$.
Then the modules $P^{\hat{\mathtt{R}}_w}(x)$, $x\in
\mathtt{R}_w$, are exactly the indecomposable projective-injective
modules in the parabolic subcategory $\mathcal{O}_0^S$ (in the sense
of \cite{RC}) of $\mathcal{O}_0$ (\cite[Remark~14]{MS}). Amongst the
indecomposable projective-injective modules in $\mathcal{O}_0^S$
there is, due to \cite[3.1]{IS}, a special one which is obtained as
a translation of some simple projective module (out of possibly
several walls). Since translation to walls maps simple modules to
simples or zero, the special module, call it $P$, is thus obtained
as a translation of some $L(x)$ for some $x\in \mathtt{R}_w$.

From \cite[Theorem~1]{KMS} it further follows that translating $P$
and taking appropriate direct summands, we will finally get all
$P^{\hat{\mathtt{R}}_w}(x)$, $x\in \mathtt{R}_w$. This implies the
existence of an indecomposable projective functor $\theta_y$ such
that the module $\theta_y L(\overline{w})$ contains
$P^{\hat{\mathtt{R}}_w}(\overline{w})$ as a direct summand (see
\cite[5.1]{MS}). By \cite[Theorem~18]{MS}, the above restriction
that the right cell should contain $w'_0w_0$ is in fact superfluous.
Moreover, from \cite[Theorem~18]{MS} it also follows that the module
$P^{\hat{\mathtt{R}}_w}(\overline{w})$ is an injective object in
$\cO^{\hat{\mathtt{R}}_w}$ (and so the same holds for any
$P^{\hat{\mathtt{R}}_w}(x)$, $x\in \mathtt{R}_w$).

Consider now $\theta_y P^{\hat{\mathtt{R}}_w}(e)\cong
P^{\hat{\mathtt{R}}_w}(y)$. As $P^{\hat{\mathtt{R}}_w}(\overline{w})$
is both projective and injective, from  Lemma~\ref{lemma5} it follows that
$P^{\hat{\mathtt{R}}_w}(\overline{w})$ must be a direct summand of
$P^{\hat{\mathtt{R}}_w}(y)$. As $P^{\hat{\mathtt{R}}_w}(y)$ is
indecomposable, this forces $P^{\hat{\mathtt{R}}_w}(y)\cong
P^{\hat{\mathtt{R}}_w}(\overline{w})$, $y=\overline{w}$,
and finally $[P^{\hat{\mathtt{R}}_w}(e):L(\overline{w})]=1$.
\end{proof}

From Lemma~\ref{lemma5} and Lemma~\ref{lemma6} it follows that for
any $x\in \mathtt{R}_w$ we have $\theta_x
P^{\hat{\mathtt{R}}_w}(e)\cong \theta_x L(\overline{w})$. This
finally proves the key statement and at the same time completes the
proof of Proposition~\ref{prop3} and Theorem~\ref{tmain}.
\end{proof}

\begin{remark}\label{parabolics}
{\rm
Let $w\in S_n$ be such that the right cell $\mathtt{R}_w$ contains
the element $w'_0w_0$ for some Young subgroup $W'$ of $S_n$. Then
$\mathcal{O}_0^{\hat{\mathtt{R}}_w}$ is the regular block of the
parabolic category $\mathcal{O}$ (in the sense of \cite{RC})
associated with $W'$. The elements $f_x$, $x\not\leq_{\mathtt{R}}w$,
form a basis of a submodule $N$ of $\mathbb{C}[S_n]$. The
quotient $\mathbb{C}[S_n]/N$ is isomorphic to the induced sign module
$\mathbb{C}[S_n]\otimes_{\mathbb{C}[W]}\mathrm{sign}$
(see \cite[6.2.1]{MS} for details) with the classes of the elements
$f_x$, $x\leq_{\mathtt{R}}w$ forming a basis. Alternatively, the elements $f_x$, $x\leq_{\mathtt{R}}w$,
form a basis of a submodule of $\mathbb{C}[S_n]$ which is isomorphic
to the induced sign module.
}
\end{remark}

\section{Proof of Theorem~\ref{tmain2}}\label{s25}

Using \eqref{eq01} and \eqref{eq1} we interpret
$h^i_{(x,y)}=\varphi([\theta_y L(x^{-1})])$ for each $i\in I$
and $(x,y)\in\mathtt{R}_i\times \mathtt{R}_i$. Let $i\in I$ be fixed.
Because of Proposition~\ref{prop3} and the definition of $g_w$'s,
to prove Theorem~\ref{tmain2} it is enough to show that every
$\theta_y L(x^{-1})$ is a projective-injective module in
$\mathcal{O}_0^{\hat{\mathtt{R}}_{x^{-1}}}$. In the case $x=\overline{y}$
this follows from the Key statement of Section~\ref{s2}.

Let now $x\in \mathtt{R}_i$ be arbitrary. As $x$ and $\overline{y}$
belong to the same right cell, the elements $x^{-1}$ and $\overline{y}$
belong to the same left cell. Let $\mathscr{A}$ and $\mathscr{B}$
denote the additive categories of projective-injective modules in
$\mathcal{O}_0^{\hat{\mathtt{R}}_{\overline{y}}}$ and
$\mathcal{O}_0^{\hat{\mathtt{R}}_{x^{-1}}}$ respectively.
In \cite[Section~5]{MS} it was shown that
there exists an equivalence
$\mathrm{F}:\mathscr{A}\rightarrow \mathscr{B}$ which commutes with
projective functors and satisfies
$\mathrm{F}(P^{\hat{\mathtt{R}}_{\overline{y}}}(\overline{y}))=
P^{\hat{\mathtt{R}}_{x^{-1}}}(x^{-1})$.

Let $\overline{\mathscr{A}}$ and $\overline{\mathscr{B}}$ denote
the full subcategories of respectively  $\mathcal{O}_0^{\hat{\mathtt{R}}_{\overline{y}}}$
and $\mathcal{O}_0^{\hat{\mathtt{R}}_{x^{-1}}}$ which consist
of all modules $X$ having a two step presentation $M_1\to M_0\to X\to 0$,
where $M_1,M_0\in \mathscr{A}$ or $M_1,M_0\in \mathscr{B}$ respectively.
Then $\mathrm{F}$ extends, in the obvious way, to an equivalence
$\overline{\mathrm{F}}:\overline{\mathscr{A}}\rightarrow
\overline{\mathscr{B}}$ which commutes with projective functors.

Let $\overline{L(\overline{y})}$ denote the quotient
of $P^{\hat{\mathtt{R}}_{\overline{y}}}(\overline{y})$
modulo the trace of all modules from $\mathscr{A}$ in the radical of
$P^{\hat{\mathtt{R}}_{\overline{y}}}(\overline{y})$.
Define $\overline{L(x^{-1})}$ analogously. Then
$\overline{L(\overline{y})}$ has simple top $L(\overline{y})$ and all
other subquotients of $\overline{L(\overline{y})}$ are of the form
$L(z)$, where  $z<_{\mathtt{R}}\overline{y}$.
Analogously $\overline{L(x^{-1})}$ has simple top $L(x^{-1})$ and all
other subquotients of $\overline{L(x^{-1})}$ are of
the form $L(z)$, where $z<_{\mathtt{R}}x^{-1}$.
From the above construction we have
$\overline{\mathrm{F}}(\overline{L(\overline{y})})=
\overline{L(x^{-1})}$. Further
$\theta_y \overline{L(\overline{y})}=\theta_y L(\overline{y})$ by
Lemma~\ref{lemma5}. Analogous arguments imply
$\theta_y \overline{L(x^{-1})}=\theta_y L(x^{-1})$. Adding everything up
we have
\begin{displaymath}
\theta_y L(x^{-1})=\theta_y \overline{L(x^{-1})}=
\theta_y \overline{\mathrm{F}}(\overline{L(\overline{y})})=
\overline{\mathrm{F}}(\theta_y \overline{L(\overline{y})})=
\overline{\mathrm{F}}(\theta_y L(\overline{y}))=
\mathrm{F}(\theta_y L(\overline{y})).
\end{displaymath}
Hence $\theta_y L(x^{-1})=\mathrm{F}(\theta_y L(\overline{y}))$
is a projective-injective module in
$\mathcal{O}_0^{\hat{\mathtt{R}}_{x^{-1}}}$. The claim follows.

\section{An application to Kostant's problem}\label{s4}

The core object $\Delta^{\hat{\mathtt{R}}_w}(e)$ of our study in
Section~\ref{s2} has an unexpected application to the so-called Kostant's
 problem from \cite{Jo}; see also \cite[Kapitel~6]{Ja}.

Let $\mathfrak{g}$ be a complex reductive finite-dimensional Lie algebra. For
every $\mathfrak{g}$-module $M$ we have the bimodule $\mathscr{L}(M,M)$ of all
$\mathbb{C}$-linear endomorphisms of $M$ on which the adjoint action of the
universal enveloping algebra $U(\mathfrak{g})$ is locally finite.
(That means
any vector $f\in\mathscr{L}(M,M)$ lies inside a finite dimensional subspace
which is stable under the adjoint action defined as $x.f(m)=x(f(m))-f(xm)$ for
$x\in\mg$, $m\in M$).
Initiated by \cite{Jo}, {\it Kostant's problem}  became
the standard terminology  for the following question concerning an arbitrary
$\mathfrak{g}$-module $M$:
\vspace{1mm}

{\em
Is the natural injection $U(\mathfrak{g})/\mathrm{Ann}(M)\hookrightarrow
\mathscr{L}(M,M)$
surjective?
}
\vspace{2mm}

Although there are several classes of modules for which the answer is known to
be positive (see \cite{Jo}, \cite{Ma}, \cite{MS} and references therein), a
complete answer to this problem seems to be far away - the problem is not even
solved for simple highest weight modules. In \cite[9.5]{Jo} an example of a
simple highest weight module in type $B_2$, for which the answer is negative is
mentioned (for details see \cite[11.5]{MS}). In this section we use the module
$\Delta^{\hat{\mathtt{R}}_xw}(e)$ to construct another example in type $A_3$,
which disproves a general belief that the answer to Kostant's problem is
positive for simple highest weight modules in type  $A$ (this belief was based
on \cite[9.1]{Jo} and further strengthened by \cite[Theorem~60]{MS}).

Let $n=4$ and $r=(12)$, $s=(23)$, $t=(34)$ be the standard Coxeter generators
of $S_4$. Consider $w=rt=\overline{w}$. In this case  we have
$\mathtt{R}_w=\{rt,rts\}$ and
$\hat{\mathtt{R}}_w=\{rt,rts,t,ts,tsr,r,rs,rst,e\}$. We consider the graded
version of $\mathcal{O}$ as worked out in \cite{St2}. Using \cite[Appendix]{St}
one computes that the module $N=\Delta^{\hat{\mathtt{R}}_w}(e)$ has the
following graded filtration (resp. socle or radical filtration), where we abbreviate $L(x)$ simply by $x$:
\begin{displaymath}
\begin{array}{ccccc}
 &&&e& \\
N= &&r&&t \\
 &&&rt&
\end{array}
\end{displaymath}

\begin{lemma}\label{lemma10}
$\mathrm{Ann}(L(rt))=\mathrm{Ann}(N)$
\end{lemma}

\begin{proof}
Let $Y_r$ and $Y_t$ denote some non-zero elements from the negative root spaces
corresponding to $r$ and $t$ respectively. Let further $U'$ be the localization
of $U(\mathfrak{sl}_4)$ with respect to the multiplicative set
$\{Y_r^iY_t^j\vert i,j\geq 0\}$.  As $rt>r$ and $rt>t$ with respect to the
Bruhat order, both $Y_r$ and $Y_t$ act injectively on $L(rt)$. Hence $L(rt)$
will be the simple socle of the $\mathfrak{sl}_4$-module
$N'=U'\otimes_{U(\mathfrak{sl}_4)}L(rt)$. As $t>e$ it is further easy to see
(for example using the results of \cite[Section~4]{KM}) that $N$ is a submodule
of $N'$. Hence the statement of the lemma would follow if we would prove that
$\mathrm{Ann}(L(rt))=\mathrm{Ann}(N')$. In fact, as $L(rt)\subset N'$, we have
only to prove that $\mathrm{Ann}(L(rt))\subset \mathrm{Ann}(N')$. This however,
follows from the following statement:

\begin{lemma}\label{l15}
Let $\mathfrak{g}$ be a semi-simple finite-dimensional Lie algebra, $0\neq x\in
\mathfrak{g}$ some root vector, and $M$ a $\mathfrak{g}$-module on which  $x$
acts injectively. Let $U'$ be the localization of $U(\mathfrak{g})$ with
respect to the powers of $X$. Then $\mathrm{Ann}(M)\subset \mathrm{Ann}(M')$,
where $M'=U'\otimes_{U(\mathfrak{g})}M$.
\end{lemma}

\begin{proof}
The set $X:=\{x^i\mid i\geq0\}$ is an Ore set in $\cU(\mathfrak{g})$ with
$X\cap\mathrm{Ann}(M)=\emptyset$ by hypothesis. So
$U'\mathrm{Ann}(M)=\mathrm{Ann}(M)U'$ is a proper ideal in $U'$. This means
$\mathrm{Ann}(M)M'=\mathrm{Ann}(M)U'M=U'\mathrm{Ann}(M)M=\{0\}$.
This completes the proof.
\end{proof}
The proof of Lemma~\ref{lemma10} is now complete.
\end{proof}

\begin{lemma}\label{lemma11}
\begin{enumerate}[(a)]
\item \label{lemma11.1} The module $\theta_t\theta_s\theta_r N$ has
the following graded filtration:
\begin{displaymath}
\begin{array}{ccccc}
&rst&&& \\
rs&&rt&& \\
rst&tsr&trs&r&\\
&&rt&&
\end{array}
\end{displaymath}
\item \label{lemma11.2} The module $\theta_t\theta_s\theta_r L(rt)$ is
a submodule of the module $\theta_t\theta_s\theta_r N$ and has
the following graded filtration:
\begin{displaymath}
\begin{array}{ccccc}
&&rt&& \\
&tsr&trs&r&\\
&&rt&&
\end{array}
\end{displaymath}
\end{enumerate}
\end{lemma}

\begin{proof}
This is verified by direct computations.
\end{proof}

\begin{theorem}\label{tkostant}
Kostant's problem has a negative answer for $L(rt)$.
\end{theorem}

\begin{proof}
As $N$ is a quotient of the dominant Verma module, Kostant's problem has a
positive solution for $N$ by \cite[6.9]{Ja}. Hence $\mathscr{L}(N,N)=
U(\mathfrak{sl}_4)/\mathrm{Ann}(N)$. By Lemma~\ref{lemma10}, we have
$\mathrm{Ann}(N)=\mathrm{Ann}(L(rt))$ and hence we also have
$U(\mathfrak{sl}_4)/\mathrm{Ann}(N)=U(\mathfrak{sl}_4)/\mathrm{Ann}(L(rt))$.
From Lemma~\ref{lemma11}\eqref{lemma11.1} we obtain that $\dim
\mathrm{Hom}_{\mathcal{O}}(N,\theta_t\theta_s\theta_r N)=0$ (as for the top
$L(e)$ of $N$ we have $[\theta_t\theta_s\theta_r N:L(e)]=0$), while
$\dim\mathrm{Hom}_{\mathcal{O}}(L(rt),\theta_t\theta_s\theta_r L(rt))\neq 0$
by  Lemma~\ref{lemma11}\eqref{lemma11.2} (as
$L(rt)$ obviously occurs in the socle of $\theta_t\theta_s\theta_r L(rt)$).
This implies $\mathscr{L}(N,N)\not=\mathscr{L}(L(rt),L(rt))$, which, in turn,
yields $\mathscr{L}(L(rt),L(rt))\not= U(\mathfrak{sl}_4)/\mathrm{Ann}(L(rt))$.
The claim follows.
\end{proof}


\end{document}